\documentclass{amsart}
\usepackage{amssymb,amscd}
\usepackage{amsthm}

\usepackage[all]{xy}

\usepackage{fullpage}

\def\torus{{\bf T}}
\def\integer{{\bf Z}}
\def\bbZ{{\bf Z}}
\def\reals{{\bf R}}
\def\complex{{\bf C}}

\def\alg{{\mathcal A}}
\def\bdd{{\mathcal B}}
\def\cpt{{\mathcal K}}
\def\cS{{\mathcal S}}
\def\cL{{\mathcal L}}
\def\cH{{\mathcal H}}
\def\CT{{\textrm CT}}

\def\ft#1{\widehat{#1}}

\newcommand\Aut{\textrm{Aut}}
\newcommand\Ca{$C^*$-algebra}

\newcommand{\thm}[1]{\advance\count32 by 1\bigskip\noindent{\bf #1 \the\count31.\the\count32.}}

\theoremstyle{plain}
\newtheorem{theorem}{Theorem}[section]
\newtheorem{lemma}[theorem]{Lemma}
\newtheorem{proposition}[theorem]{Proposition}
\newtheorem{corollary}[theorem]{Corollary}

\theoremstyle{definition}

\newtheorem{definition*}{Definition}
\newtheorem{example}[theorem]{Example}
\theoremstyle{remark}
\newtheorem{remark}[theorem]{Remark}

\newtheorem{remarks*}{Remarks}

\begin{document}

\title[Parametrised strict deformation quantization of $C^*$-algebras]
{Parametrised strict deformation quantization of $C^*$-bundles and Hilbert $C^*$-modules}

\author[KC Hannabuss]{Keith C. Hannabuss}

\address[Keith Hannabuss]{
Mathematical Institute, \\
24-29 St. Giles', \\
Oxford, OX1 3LB, and \\
Balliol College, \\
Oxford, OX1 3BJ,\\
England}
\email{kch@balliol.ox.ac.uk}

\author[V Mathai]{Varghese Mathai}

\address[Varghese Mathai]{
Department of Pure Mathematics,\\
University of Adelaide,\\
Adelaide, SA 5005, \\
Australia}
\email{mathai.varghese@adelaide.edu.au}

\maketitle

\begin{abstract}
In this paper, we review the parametrised strict deformation quantization
of $C^*$-bundles obtained in a previous paper, and give more examples and applications
of this theory.
In particular, it is used here to classify $H_3$-twisted
noncommutative torus bundles over a locally compact space. This is
extended to the case of general torus bundles and their
parametrised strict deformation quantization. Rieffel's basic construction 
of an algebra deformation can be mimicked to deform a monoidal category, 
which deforms not only algebras but also modules. As a special case, we consider the
parametrised strict deformation quantization of Hilbert $C^*$-modules
over $C^*$-bundles with fibrewise torus action.
\end{abstract}

\bigskip

\begin{center}
\it Dedicated to Alan Carey, on the occasion of his 60$\,^{th}$ birthday
\end{center}

\bigskip

\section*{Introduction}
Parametrised strict deformation quantization of $C^*$-bundles
was introduced by the authors in an earlier paper \cite{HM}, as a generalization
of Rieffel's strict deformation quantization of $C^*$-algebras \cite{Rieffel1, Rieffel2}.
The particular version of Rieffel's theory that was generalized in \cite{HM} was due to
Kasprzak \cite{Kasp} and Landstad \cite{Land1, Land2}.
The results in \cite{HM} were used to classify noncommutative principal torus bundles as defined by
Echterhoff, Nest, and Oyono-Oyono \cite{ENOO}, as parametrised strict deformation quantizations
of principal torus bundles.
These arise as special cases of the continuous fields of noncommutative tori that appear
as T-duals to spacetimes with background H-flux in \cite{MR, MR2}.

We review in \S\ref{sect:param}, the construction of the parametrized deformation quantization
from our earlier paper \cite{HM} and give more examples in sections \ref{sect:param}, \ref{sect:examples}, and applications 
in section \ref{sect:classify}, of this construction.
For example, we generalize the main
application of our results in \cite{HM} (as well as those in
\cite{ENOO}). 
More precisely, suppose that $A(X)$ is a $C^*$-bundle over a
locally compact space $X$ with a fibrewise action of a torus
$T$, and that $A(X) \rtimes T \cong \CT(X, H_3)$, where $\CT(X,
H_3)$ is a continuous trace algebra with spectrum $X$ and
Dixmier-Douady class $H_3 \in H^3(X; \integer)$. We call such
$C^*$-bundles, $H_3$-twisted NCPT bundles over $X$. Our first
main result is that any $H_3$-twisted NCPT bundle $A(X)$ is
equivariantly Morita equivalent to the parametrised deformation
quantization of the continuous trace algebra $$\CT(Y,
q^*(H_3))_{\sigma},$$ where $q: Y \to X$ is a principal torus
bundle with Chern class  $H_2 \in H^2(X;
H^{1}(T;\integer))$, and $\sigma\in C_b(X, Z^2(\ft{T},\torus))$
a  defining deformation such that $[\sigma] = H_1 \in H^1(X;
H^{2}(T;\integer))$. This enables us to prove in section \ref{sect:T-duality} that the
continuous trace algebra $${\textrm CT}(X\times T, H_1 + H_2
+H_3),$$ with Dixmier-Douady class $H_1+H_2+H_3 \in
H^3(X \times T; \integer), \,\,H_j \in H^j(X; H^{3-j}(T;\integer))$,
has an action of the vector group $V$ that is the universal
cover of the torus $T$, and covering the $V$-action on $X \times T$.  Moreover the crossed product can be
identified up to $T$-equivariant Morita equivalence,
$$ \CT(X\times T, H_1 + H_2 +H_3) \rtimes V \cong \CT(Y, q^*(H_3))_{\sigma}.$$
That is, the T-dual of $(X\times T, H_1 + H_2 +H_3)$ is the parametrised strict deformation quantization
of $(Y, q^*(H_3))$ with deformation parameter $\sigma,\, [\sigma]=H_1$. From this we obtain the explicit
dependence of the K-theory of $\CT(Y, q^*(H_3))_{\sigma}$ in terms of the deformation parameter.

In section 6, we extend this to the case of general torus
bundles and their noncommutative parametrised strict
deformation quantizations.
It has proved a useful principle that deformation of an algebraic structure should be
viewed within the context of a deformation of an appropriate  category, see, for example,
\cite[Introduction]{B+L}.  
Following this philosophy, we show in Section 7 how Rieffel
deformations can be regarded as monoidal functors which allow
us to deform modules as well as algebras.
 In the last section,  we construct the parametrised strict deformation quantization
of Hilbert $C^*$-modules over $C^*$-bundles with fibrewise
torus action directly.

\section{$C^*$-bundles and fibrewise smooth $*$-bundles}
\label{sec:prelim}
We begin by recalling the notion of $C^*$-bundles over $X$ and then introduce
the special
case of $H_3$-twisted noncommutative principal bundles. Then we discuss the fibrewise smoothing
of these, which is used in parametrised Rieffel deformation later on.

Let $X$ be a locally compact Hausdorff space and let $C_0(X)$ denote
the  {\Ca} of continuous functions on $X$ that vanish at infinity.
A $C^*$-\emph{bundle} $A(X)$ over $X$ in the sense of \cite{ENOO}
is  exactly a $C_0(X)$-algebra
in the sense of Kasparov \cite{Kas88}.  That is, $A(X)$ is a
{\Ca} together with a non-degenerate $*$-homomorphism
$$\Phi: C_0(X)\to ZM(A(X)),$$
called the {\emph{ structure map}}, where
$ZM(A)$ denotes the center of the multiplier algebra $M(A)$ of $A$.
The {\emph{ fibre}}   over  $x\in X$ is then  $A(X)_x=A(X)/I_x$, where
$$
I_x=\{\Phi(f)\cdot a;\, a\in A(X)\, {\textrm{and}}\, f\in C_0(X)\, {\textrm{such that}}\, f(x)=0\},
$$ 
and the canonical quotient map $q_x:A(X)\to
A(X)_x$ is called the {\emph{ evaluation map}}  at $x$.

Note that this definition does not require local triviality of the
bundle, or even for the fibres of the bundle to be isomorphic to one
another.

Let $G$ be a locally compact group. One says that there is a {\emph{
fibrewise action}} of $G$ on a $C^*$-bundle $A(X)$ if there is a
homomorphism $\alpha: G \longrightarrow  \Aut(A(X))$ which is
$C_0(X)$-linear in the sense that
$$
\alpha_g(\Phi(f) a) = \Phi(f)(\alpha_g(a)), \qquad \forall g\in G, \,
a\in A(X),\, f \in C_0(X).
$$
This means that $\alpha$ induces an action $\alpha^x$ on the fibre
$A(X)_x$ for all $x\in X$.

The first observation is that if
$A(X)$ is a $C^*$-algebra bundle over $X$ with a fibrewise action $\alpha$
of a {\emph{ Lie group}} $G$, then there is a canonical {\emph{ smooth $*$-algebra bundle}}
over $X$. We recall its definition from \cite{Co80}.
A vector $y\in A(X)$ is said to be a {\emph{ smooth vector}} if the map
$$
G \ni g \longrightarrow \alpha_g(y) \in A(X)
$$
is a smooth map from $G$ to the normed vector space $A(X)$. Then
$$
\alg^\infty(X) = \{ y\in A(X)\, | \, y\, {\textrm{ is  a smooth vector}} \}
$$
is a $*$-subalgebra of $A(X)$ which is norm dense in $A(X)$. Since $G$
acts {\emph{ fibrewise}} on $A(X)$, it follows that $\alg^\infty(X) $ is again a
$C_0(X)$-algebra which is {\emph{ fibrewise smooth}}.

Let $T$ denote the torus of dimension $n$.  We define a  {\emph{ $H_3$-twisted noncommutative principal  $T$-bundle}}
(or {\emph{ $H_3$-twisted NCP $T$-bundle}}) over $X$ to be a separable  $C^*$-bundle
$A(X)$  together with a fibrewise  action $\alpha:T\to\Aut(A(X))$
such that  there is a Morita equivalence,
$$A(X)\rtimes_{\alpha}T\cong \CT(X, H_3),$$
as $C^*$-bundles over $X$, where $ \CT(X, H_3)$ denotes the
continuous trace $C^*$-algebra with spectrum equal to $X$ and
Dixmier-Douady class equal to $H_3 \in H^3(X, \integer)$. 

If $A(X)$ is a $H_3$-twisted NCP $T$-bundle over $X$, then we call $\alg^\infty(X)$
a {\emph{ fibrewise smooth $H_3$-twisted noncommutative principal  $T$-bundle}}
(or {\emph{ fibrewise smooth $H_3$-twisted NCP $T$-bundle}}) over $X$. In this paper, we are able
to give a complete classification of fibrewise smooth $H_3$-twisted NCP $T$-bundles over $X$
via a parametrised version of Rieffel's theory of strict deformation quantization as derived in \cite{HM}.

\section{Parametrised strict deformation quantization of $C^*$-bundles with fibrewise action of $T$}
\label{sect:param}

In a nutshell, parametrised strict deformation quantization
is a functorial extension of Rieffel's strict deformation quantization
from algebras $A$, to $C(X)$-algebras $A(X)$, and in particular
to $C^*$-bundles over $X$.
Unlike Rieffel's deformation theory \cite{Rieffel1}
the version in \cite{HM} starts
with multipliers via the Landstad--Kasprzak
approach.  Here we review the theory for $C^*$-bundles over $X$.

Let $A(X)$ be a $C^*$-algebra bundle over $X$ with a fibrewise action $\alpha$
of a torus $T$. Let  $\sigma\in C_b(X, Z^2(\ft{T},\torus))$ be a deformation parameter.
Then we define the parametrised strict deformation quantization of $A(X)$, denoted
$A(X)_\sigma$ as follows. We have the direct sum decomposition,
$$
\begin{array}{lcl}
A(X)  &\cong & \widehat\bigoplus_{\chi \in \hat T} A(X)_\chi
 \\
\phi(x) &= &\sum_{\chi \in \hat T} \phi_\chi(x)
\end{array}
$$
for $x\in X$, where for $\chi\in\widehat{T}$,
$$
A(X)_\chi:= \left\{a\in A(X) \mid
  \alpha_t(a)=\chi(t)\cdot a\,\,
 \,\,\,   \forall\, t\in T \right\}.
$$
Since $T$ acts by $\star$-automorphisms, we have
\begin{equation}
  \label{eq:Fell_bundle_algebra}
  A(X)_\chi\cdot A(X)_\eta \subseteq A(X)_{\chi\eta}
  \quad {\textrm{and}} \quad A(X)_\chi^*=A(X)_{\chi^{-1}}
  \qquad  \forall\, \chi,\eta\in\widehat{T}.
\end{equation}
Therefore the spaces \(A(X)_\chi\) for \(\chi\in\widehat{T}\) form a
{Fell bundle}~\(A(X)\) over~\(\widehat{T}\)
(see~\cite{fell_doran}); there is no continuity condition
because~\(\widehat{T}\) is discrete.
The completion of the direct sum is explained as follows.
The representation theory of $T$ shows that
\(\bigoplus_{\chi\in\widehat{T}} A(X)_\chi = A(X)^{alg}\) is a
$T$-equivariant dense subspace
of~\(A(X)\), where $T$ acts on $ A(X)_\chi$ as follows:
$\hat \alpha_t(\phi_\chi(x)) = \chi(t) \phi_\chi(x)$ for all $t\in T, \, x\in X$.
Then $ \widehat{A(X)^{alg}} = \widehat\bigoplus_{\chi \in \hat T} A(X)_\chi$
is the completion in the $C^*$-norm of $A(X)$, and is isomorphic to $A(X)$.

The product of elements in $A(X)^{alg} \subset A(X)$ then also decompose as,
$$
(\phi \psi)_\chi(x)= \sum_{\chi_1\chi_2=\chi} \phi_{\chi_1}(x)\psi_{\chi_2}(x)
$$
for $\chi_1,\chi_2, \chi \in \widehat T$. The product can be deformed by setting
$$
(\phi\star_\sigma \psi)_\chi(x) = \sum_{\chi_1\chi_2=\chi} \phi_{\chi_1}(x)\psi_{\chi_2}(x)
\sigma(x; \chi_1, \chi_2)
$$
which is associative because of the
cocycle property of $\sigma$.


We next describe the norm completion aspects. For $x\in X$, let $\cH_x$ denote the universal Hilbert space representation of the fibre $C^*$-algebra $A(X)_x$ which one obtains via the GNS theorem. Let $\cH_1 = \int_X \cH_x dx$ denote the direct integral, which is the universal Hilbert space representation of $A(X)$. By considering instead the Hilbert space $\cH=\cH_1 \otimes L^2(T) \otimes \cH_2$, where $\cH_2$ is an infinite dimensional Hilbert space, where we note that every character of $T$ occurs with infinite multiplicity in $L^2(T) \otimes \cH_2$, we obtain a $T$-equivariant embedding $\varpi:A(X) \to B(\cH)$. The equivariance means that
$$
\varpi(\phi(x)_\chi) = \varpi(\phi(x))_\chi.
$$
Now consider the action of $A(X)^{alg}$ on $\cH$ given by the deformed product $\star_\sigma$, that is, for $\phi \in A(X)^{alg}$ and $\Psi \in \cH$,
$$ (\phi\star_\sigma \Psi)_\chi(x) = \sum_{\chi_1\chi_2=\chi} \varpi(\phi_{\chi_1}(x))\Psi_{\chi_2}(x) \sigma(x; \chi_1, \chi_2).
$$
The operator norm completion of this action is the parametrised strict deformation quantization of $A(X)$, denoted by $A(X)_\sigma$.

We next consider a special case of this construction. Consider
a smooth fiber bundle of smooth manifolds,
\begin{equation}
\xymatrix{        
Z\ar[r]&Y\ar[d]^\pi\\ 
& X.
}
\end{equation}

Suppose there is a fibrewise action of a torus $T$ on $Y$. That
is, assume that there is an action of $T$ on $Y$ satisfying,
$$
\pi(t.y) = \pi(y), \qquad \forall\, t\in T, \, y\in Y.
$$
Let  $\sigma\in C_b(X, Z^2(\ft{T},\torus))$ be a deformation parameter.
$C_0(Y)$ is a $C^*$-bundle over $X$, and as above, form the
parametrised strict deformation quantization $C_0(Y)_\sigma$.

In particular, let $Y$ be a principal $G$-bundle over $X$, where
$G$ is a compact Lie group such that ${\textrm rank}(G) \ge 2$.
(e.g. $G= {\textrm SU}(n), \, n\ge 3$ or  $G= {\textrm U}(n), \, n\ge 2$).
Let  $T$ be a maximal torus in $G$ and 
$\sigma\in C_b(X, Z^2(\ft{T},\torus))$ be a deformation parameter.
Then $C_0(Y)$ is a $C^*$-bundle over $X$, and as above, form the
parametrised strict deformation quantization $C_0(Y)_\sigma$.

\section{Classifying $H_3$-twisted NCPT-bundles}
\label{sect:classify}

Here we prove another application of parametrised strict deformation quantization
cf. \S\ref{sect:param}, \cite{HM}.

Let $A(X)$ be a $H_3$-twisted NCPT-bundle over $X$. Consider
 the $C^*$-bundle over $X$, $A(X)\otimes_{C_0(X)} \CT(X, -H_3)$, which
has a fibrewise (diagonal) action of $T$, where $T$ acts trivially
on $\CT(X, -H_3)$. Then
$$
\begin{array}{lcl}
 \left(A(X)\otimes_{C_0(X)} \CT(X, -H_3)\right)\rtimes T 
&\cong&  \left(A(X) \rtimes T \right)\otimes_{C_0(X)} \CT(X, -H_3) \\
& \cong & \CT(X, H_3) \otimes_{C_0(X)} \CT(X, -H_3)\\
& \cong& C_0(X, \cpt).
\end{array}
$$

Therefore, $A(X)\otimes_{C_0(X)} \CT(X, -H_3)$ is a NCPT-bundle over $X$.
By the classification Theorem 5.1 \cite{HM}, (which in turn used the results of 
Echterhoff and Williams, \cite{EW}) we deduce that
$$
A(X)\otimes_{C_0(X)} \CT(X, -H_3) \cong C_0(Y)_\sigma.
$$
Therefore,
$$
A(X) \cong C_0(Y)_\sigma \otimes_{C_0(X)} \CT(X, H_3).
$$

\begin{lemma}\label{lem:H_3-twisted}
In the notation above,
$$
C_0(Y)_\sigma \otimes_{C_0(X)} \CT(X, H_3) \cong \CT(Y, q^*(H_3))_\sigma
$$
\end{lemma}

\begin{proof}
We use the explicit Fourier decomposition as in Example \ref{ex:torusbundledd}
and Example 6.2 \cite{HM}
 to deduce that first of all that both sides are naturally isomorphic as $T$-vector spaces,
 and also that products are compatible under the isomorphism.
\end{proof}

To summarize, we have the following main result, which follows from
Theorem 3.1 \cite{HM},  \S 4 \cite{HM}, Example \ref{ex:torusbundledd},
and the observations above.

\begin{theorem}\label{thm:classify}\footnote{This also easily follows from the well known fact
that in the Brauer group ${\textrm Br}_G(X)$   for $G$ acting trivially on  $X$, 
every element in this Brauer group factors by a product of the trivial action on A and
an action on $C_0(X,\cpt)$ (e.g. see \cite{CKRW}).
Anyway, the arguments presented here are direct and quite simple.}
Given a $H_3$-twisted NCPT-bundle $A(X)$, there is
a  defining deformation $\sigma\in C_b(X, Z^2(\ft{T},\torus))$ and a
principal torus bundle $q:Y \to X$ such that
$A(X)$ is $T$-equivariant Morita equivalent over $C_0(X)$, to the parametrised strict deformation
quantization of $\CT(Y, q^*(H_3))$  with respect to
$\sigma$, that is,
$$
A(X) \cong \CT(Y, q^*(H_3))_\sigma.
$$
Conversely, by Example \ref{ex:torusbundledd}, the parametrised strict deformation
quantization of $\CT(Y, q^*(H_3))$ is the $H_3$-twisted NCPT-bundle
$\CT(Y, q^*(H_3))_\sigma$.
\end{theorem}

\section{T-duality and K-theory}
\label{sect:T-duality}

Here we prove the result,
\begin{theorem}\label{thm:Tduality} In the notation above,
$(X\times T, H_1 + H_2 +H_3)$ and the parametrised strict deformation quantization
of $(Y, q^*(H_3))$ with deformation parameter $\sigma,\, [\sigma]=H_1$,
are T-dual pairs, where the 1st Chern class $c_1(Y) = H_2$. That is,
$$
 \CT(Y, q^*(H_3))_\sigma \rtimes V  \cong \CT(X\times T, H_1 + H_2 + H_3).
 $$
\end{theorem}

\begin{proof}
 As before, let $V$ be the vector group that is the 
 universal covering group of the torus group $T$, and the action of $V$ on the spectrum 
 factors through $T$.
 By Lemma 8.1 in \cite{ENOO}, the crossed product $\CT(X\times T, H_1 + H_2) \rtimes_\beta V$ is Morita equivalent to 
 $C_0(X,  \cpt) \rtimes_\sigma \ft{T}$, where as before, the Pontryagin dual group $ \ft{T}$ acts fibrewise on  $C_0(X, \cpt) $. 
 Setting $A(X)= C_0(X,  \cpt) \rtimes_\sigma \ft{T}$, then it is a $C^*$-bundle over $X$ with a fibrewise
 action of $T$ and by Takai duality, $A(X) \rtimes T \cong C_0(X, \cpt)$. Therefore $A(X)$ is a NCPT-bundle
 and by Theorem 5.1 in \cite{HM}, we see that there is a $T$-equivariant Morita equivalence,
  $$A(X) \sim C_0(Y)_\sigma, $$
 where the notation is as in the statement of this Theorem.
 By Lemma \ref{lem:H_3-twisted}, 
$$
 \CT(Y, q^*(H_3))_\sigma \rtimes V \cong \left(C_0(Y)_\sigma  \rtimes V\right) \otimes_{C_0(X)} \CT(X, H_3).
 $$
By Takai duality, 
 $C_0(Y)_\sigma  \rtimes V \cong \CT(X\times T, H_1 + H_2)$. Therefore
$$
 \CT(Y, q^*(H_3))_\sigma \rtimes V  \cong \CT(X\times T, H_1 + H_2 + H_3),
 $$
 proving the result.
\end{proof}

Using Connes Thom isomorphism theorem \cite{Connes} and the result above, one has

\begin{corollary}\label{thm:Ktheory}
The K-theory of $\CT(Y, q^*(H_3))_\sigma$ depends on the deformation parameter in general.
More precisely, in the notation above $ [\sigma]=H_1$, $c_1(Y) = H_2$,
$$
K_\bullet(\CT(Y, q^*(H_3))_\sigma) \cong K^{\bullet + \dim V}(X\times T, H_1 + H_2 +H_3),
$$
where the right hand side denotes the twisted K-theory.
\end{corollary}

\section{Examples}
\label{sect:examples}

\begin{example}[Noncommutative torus]
We begin by recalling the construction by Rieffel \cite{Rieffel1} realizing the smooth noncommutative
torus as a deformation quantization of the smooth functions on a torus
$T = \reals^{n}/\integer^{n}$ of dimension equal to $n$.

Recall that any translation invariant Poisson bracket on $T$  is just
$$
\{a, b\} = \sum \theta_{ij} \frac{\partial a}{\partial x_i}\frac{\partial b}{\partial x_j},
$$
for $a, b \in C^\infty(T)$, where $(\theta_{ij} )$ is a skew symmetric matrix.
The action of $T$ on itself is given by translation. The Fourier transform is an isomorphism
between $C^\infty(T)$ and $\cS(\hat T)$, taking the pointwise product on $C^\infty(T)$ to the
convolution product on $\cS(\hat T)$ and taking differentiation with respect to a coordinate function
to multiplication by the dual coordinate. In particular, the Fourier transform of the Poisson bracket gives
rise to an operation on $\cS(\hat T)$ denoted the same. For $\phi,\psi \in \cS(\hat T)$, define
$$
\{\psi, \phi\} (p)= -4\pi^2\sum_{p_1+p_2=p} \psi(p_1) \phi (p_2) \gamma(p_1, p_2)
$$
where $\gamma$ is the skew symmetric form on $\hat T$ defined by
$$
 \gamma(p_1, p_2) =  \sum \theta_{ij} \,p_{1,i}\,p_{2,j}.
$$
For $\hbar \in \reals$, define a skew bicharacter $\sigma_\hbar$ on $\hat T$ by
$$
\sigma_\hbar(p_1, p_2) = \exp(-\pi\hbar\gamma(p_1, p_2)).
$$
Using this, define a new associative product $\star_\hbar$ on $\cS(\hat T)$,
$$
(\psi\star_\hbar\phi) (p) = \sum_{p_1+p_2=p} \psi(p_1) \phi (p_2) \sigma_\hbar(p_1, p_2).
$$
This is precisely the smooth noncommutative torus $A^\infty_{\sigma_\hbar}$.

The norm $||\cdot||_\hbar$ is defined to be the operator norm for the action of $\cS(\hat T)$
on $L^2(\hat T)$ given by $\star_\hbar$. Via the Fourier transform, carry this structure
back to $C^\infty(T)$, to obtain the smooth noncommutative torus
as a strict deformation quantization of $C^\infty(T)$, \cite{Rieffel1} with respect to the
translation action of $T$.
\end{example}

\begin{example}\label{ex:torusbundledd}
We next generalize the first example above to the case of principal torus bundles $q:Y\to X$ of rank equal to $n$,
together with a algebra bundle $\cpt_P\to X$ with fibre the $C^*$-algebra of compact operators $\cpt$.
Here $P\to X$ is a principal bundle with structure group the projective unitary group ${\textrm PU}$, which acts
on $\cpt$ by conjugation.
Note that fibrewise smooth sections of $q^*(\cpt_P)$ over $Y$ decompose as a direct sum,
$$
\begin{array}{lcl}
C^\infty_{\textrm{fibre}}(Y, q^*(\cpt_P)) &= &\widehat\bigoplus_{\alpha \in \hat T} \ C^\infty_{\textrm{fibre}}(X, \cL_\alpha \otimes \cpt_P) \\
\phi &=& \sum_{\alpha \in \hat T} \phi_\alpha
\end{array}
$$
where $ C^\infty_{\textrm{fibre}}(X, \cL_\alpha \otimes E)$ is defined as the subspace of $C^\infty_{\textrm{fibre}}(Y, q^*(E))$
consisting of sections which transform under the character $\alpha \in \hat T$, and where $ \cL_\alpha$
denotes the associated line bundle $Y \times_\alpha \complex$ over $X$. That is,
$ \phi_\alpha(yt) = \alpha(t)  \phi_\alpha(y), \, \forall\, y\in Y,\, t\in T$.
The direct sum is completed in such a way that the function
$\hat T \ni \alpha \mapsto ||\phi_\alpha||_\infty \in \reals$
is in $\cS(\hat T)$.

For $\phi, \psi \in C^\infty_{\textrm{fibre}}(Y, q^*(\cpt_P))$, define a deformed product
$\star_\hbar$  
as follows.
For $y \in Y$, $\alpha, \alpha_1, \alpha_2 \in \hat T$, let
\begin{equation}\label{defaction}
(\psi\star_\hbar\phi) (y, \alpha) = \sum_{\alpha_1\alpha_2=\alpha} \psi(y, \alpha_1)
\phi (y, \alpha_2)
\sigma_\hbar(q(y); \alpha_1, \alpha_2),
\end{equation}
using the notation $ \psi(y, \alpha_1) =  \psi_{\alpha_1}(y)$ etc., and where
$\sigma_\hbar \in C(X,  Z^2(\hat T, \torus))$ is a continuous
 family of bicharacters of
$\hat T$ such that $\sigma_0 =1$. The cocycle property of $\sigma_\hbar$ ensures that
(\ref{defaction}) defines an associative product.
The construction is clearly $T$-equivariant. We denote the deformed algebra as $C^\infty_{\textrm{fibre}}(Y, q^*(\cpt_P))_{\sigma_\hbar}$.
\end{example}

\section{General torus bundles}

We extend the results of the previous sections to the case of
deformations of general torus bundles,
not just principal torus bundles.
That is, the noncommutative torus bundles (NCT-bundles) of this section
strictly include the noncommutative principal
torus bundles (NCPT-bundles) of  \cite{HM},\cite{ENOO}.

For any fibre bundle $F\rightarrow\ Y\stackrel{\xi}{\rightarrow} X$ with structure
group $G$, the action of $G$ on $F$ induces an action of $\pi_{0}(G)$ on the homology and cohomology
of $F$.  When $X$ is a connected manifold, there is a well-defined homomorphism
$\pi_{1}(X)\rightarrow\pi_{0}(G)$ that gives each homology or cohomology group of $F$ the
structure of a $\bbZ\pi_{1}(X)$-module.  Now suppose that $F$ is a torus $T$ and
$G={\textrm Diff}(T)$.  It is well known that
$\pi_{0}(G)\cong {\textrm GL}(n,\bbZ)$, where $n=\dim(T)$, and that the $\pi_{0}(G)$-action on $H_{1}(T)$
may be identified with the
natural action of ${\textrm GL}(n,\bbZ)$ on $\widehat T$.  Given any representation
$\rho:\pi_{1}(X)\rightarrow {\textrm GL}(n,\bbZ)$, we let $\bbZ^{n}_{\rho}$ denote the
corresponding $\bbZ\pi_{1}(B)$-module. We will make use of the following
proposition, which is well known, and explicitly stated in the appendices of \cite{Kahn}.

\begin{proposition}
Assume that $X$ is a compact, connected manifold, and choose any
representation $\rho:\pi_{1}(X)\rightarrow {\textrm GL}(n,\bbZ)$.  Then there is a natural, bijective
correspondence between the equivalence classes of torus bundles over $X$ inducing the
module structure $\bbZ^{n}_{\rho}$ on $H_{1}(T)$ and the elements of
$H^{2}(X;\bbZ^{n}_{\rho})$, the second cohomology group of $X$ with local coefficients
$\bbZ^{n}_{\rho}$.\hfill$\square$
\end{proposition}

\begin{remark}
We call the cohomology class corresponding to the symplectic torus bundle  $\xi$ the
\emph{characteristic class} of $\xi$ and denote it by $c(\xi) \in H^{2}(X;\bbZ^{n}_{\rho})$.
Then the characteristic class $c(\xi)$
vanishes if and only if $\xi$ admits a section.  When the representation $\rho$ is trivial,
then $\xi$ is a principal torus bundle and the characteristic class reduces to the first Chern class.
$c(\xi)=c_1(\xi)=0$ if and only if $\xi$ is trivial.
\end{remark}

Let $X$ be compact and $T\rightarrow Y\stackrel{\xi}{\rightarrow} X$ be a torus bundle over $X$.
Let $\Gamma \rightarrow \widehat X\stackrel{\eta}{\rightarrow} X$ denote the universal cover of
$X$, and consider the lifted torus bundle $T\rightarrow \eta^*(Y)\stackrel{\eta^*\xi}{\rightarrow} \widehat X$.
Since $\widehat X$ is simply-connected, it is classified by a characteristic class in  $H^{2}(\widehat X;\bbZ^{n})$,
so that $T\rightarrow \eta^*(Y)\stackrel{\eta^*\xi}{\rightarrow} \widehat X$ is a principal torus bundle.
Let  $\sigma\in C_b(X, Z^2(\ft{T},\torus))$ be a deformation parameter.
Then $C_0(\eta^*(Y))$ is a $C^*$-bundle over $\widehat X$, and as before, form the
parametrised strict deformation quantization $C_0(\eta^*(Y))_{\widehat\sigma}$, where
$\widehat\sigma$ is the lift of $\sigma$ to $\widehat X$. 
Since $ \eta^*(Y)$ is the total space of a principal $\Gamma$-bundle over the compact space $Y$,
the action of $\Gamma$ on $ \eta^*(Y)$
is free and proper, so the action of $\Gamma$ on $C_0(\eta^*(Y))_{\widehat\sigma}$ is also proper in the
sense of Rieffel \cite{Rieffel4}. In particular, the fixed-point algebra 
 $C_0(\eta^*(Y))_{\widehat\sigma}^\Gamma$ makes sense and is the parametrised strict deformation
 quantization of $C_0(Y)$. We define this to be a {\emph{noncommutative torus bundle}} (NCT-bundle) over $X$.
 In particular, they strictly include noncommutative {principal} torus bundles (NCPT-bundles) over $X$.
 NCT-bundles are classified by a representation $\rho:\pi_{1}(X)\rightarrow {\textrm GL}(n,\bbZ)$ together with
 a cohomology class with local coefficients, $H_2 \in  H^{2}(X;\bbZ^{n}_{\rho})$
 and a deformation parameter $\sigma\in C_b(X, Z^2(\ft{T},\torus))$.

 Finally, we define a $H_3$-twisted NCT-bundle over $X$ to be a  parametrised strict deformation
 quantization of $\CT(Y, \xi^*(H_3))$, denoted by  $\CT(\eta^*(Y), \eta^*\xi^*(H_3) )_{\widehat\sigma}^\Gamma.$


\section{Deformations of monoidal categories} \label{sec:category}
Rieffel's strict deformation theory modifies the multiplication
on an algebra.  As noted in the introduction it is a useful
principle that deformation should involve not only an algebra,
but a category in which the algebra is but one object.
Fortunately Rieffel's strict C$^*$-algebra deformation can be
extended to give a functor changing the tensor product in a
monoidal category of Fr\'echet $V$-modules for a vector group $V$.
Another motivation for this stems from the way in which nonassociative
crossed products could be understood naturally in the context
of a monoidal category, \cite{BHM}. A similar interpretation
of noncommutative crossed products could make it easier to
unify the two examples.
Let $X$ be a locally compact Hausdorff space and $V$ an abelian
group. We start with Banach $V$-modules (in which $V$ acts by
isometries), with a commuting action of $C_0(X)$. Each module
has a dense submodule of $V$-smooth vectors, on which $V$ and $C_0(X)$
still act and which can be given the structure of a Fr\'echet
space.
Consider the strict symmetric monoidal category of these
$C_0(X)$-$V$-{\sf mod} of smooth $C_0(X)\rtimes V$-modules with
tensor product $\otimes_0 = \otimes_{C_0(X)}$ as in \cite{Rieffel5},
and unit object $C_0(X)$ (with the multiplication action of $C_0(X)$ and
trivial action of $V$). We shall suppose also that we have a
symmetric bicharacter $e = e^{iB}: V\times V \to C_0(X)$, and a
$B$-skew-adjoint automorphism $J$ of $V$.
The monoidal functor ${\mathcal D}_J$ to a braided category $C_0(X)-(V,J)$-{\sf
mod} acts as the identity on objects and morphisms, but gives a
braided tensor product $\otimes_J$ and (assuming the integral
well-defined) with, for objects $\alg$ and $\bdd$, the
consistency map $c_J:\alg \otimes_J \bdd \to \alg\otimes\bdd$,
taking $x\in \alg$, $y\in \bdd$
$$
c_J(x\otimes_J y) = \int_{V\times V} e(u,v)((Ju).x) \otimes_0 (v.y) \,du\,dv,
$$
where $e$ acts on the tensor product as an element of
$C_0 (X)$. (It follows from the discussion in \cite{Rieffel1,HM} that
$c_J$ is invertible.) Essentially all the technical estimates needed
to show that this is well-defined, and to derive its properties have
already been provided by Rieffel in \cite[Chapter I]{Rieffel1}
using a simplified version of H\"ormander's partial integration
technique applied to smooth maps from a vector group to a
Fr\'echet space. In a Hilbert $V$ module $M$ the smooth vectors
$M^\infty$ for the action provide a Fr\'echet space, and the
rest is done as in \cite{Rieffel1}.
(Within the tensor product of two $V$-modules $M_1$ and $M_2$
is the dense subspace of smooth vectors $(M_1\otimes
M_2)^\infty$, which contains the tensor product
$M^\infty_1\otimes M_2^\infty$. We have just mimicked the
constructions of \cite[Chapter 2]{Rieffel1} to deform the tensor product,
rather than an algebra product.)
One can check that this is $C_0(X)$-linear
(due to the triviality of the action of $V$ on $C_0(X)$) and
is compatible with strict associativity \cite[Theorem
2.14]{Rieffel1}, and $C_0(X)$ as unit object. (For non-vanishing
$H_0$, we would instead need consistency with the
new associativity map in the usual hexagonal diagram.)
Since $V$ is abelian it has the tensor product action $\Delta$ of $V$,
(which changes to $c_J^{-1}\circ\Delta \circ c_J$in the new monoidal category).
Some care is needed because the asymmetry in $c_J$ means that
deformed category is braided. Assuming that one started with
a trivial braiding given by the flip $\Psi_0: m\otimes_0 n \to n\otimes_0 m$,
one obtains $\Psi_J = c_J^{-1}\Psi_0 c_J: M\otimes_J N \to N\otimes_J
M$. Since $\Psi_0^2 = 1$, we automatically have $\Psi_J^2 = 1$,
so that $\Psi_J$ is a symmetric braiding. More generally
$\Psi_J$ and $\Psi_0$ satisfy the same polynomial identities:
if $\Psi_0$ is of Hecke type then so is $\Psi_J$.
If an object $\alg$ has a multiplication morphism
$\mu:\alg\otimes_0\alg \to \alg$, then the morphism property
ensures that $V$ acts by automorphisms, and the deformed
multiplication is $\mu\circ c_J: \alg\otimes_J\alg \to \alg$,
or, using $\star_J$ and $\star$ for the deformed and undeformed
multiplications,
$$
x\star_J y = \int_{V\times V} e(u,v)((Ju).x) \star (v.y) \,du\,dv,
$$
which is the Rieffel deformed product \cite{Rieffel1}. Due to
the braiding a commutative algebra product $\mu$ (satisfying
$\mu\circ\Psi = \mu$) turns into a braided commutative, but
noncommutative, product with $\mu_J\circ\Psi_J = \mu_J$, as one
would expect for a map transforming classical to quantum
theory.
We can similarly deform $\alg$-modules and bimodules. For
example, an action $\alpha:\alg\otimes_0 M \to M$ can be
deformed to $\alpha_J:\alg\otimes_J M \to M$
$$
\alpha_J(a)[m] = \int_{V\times V} e(u,v)\alpha((Ju).a)[v.m] \,du\,dv.
$$ One can follow the same strategy as in \cite{BHM} and study the effects on compact operators, crossed products etc, but we shall content ourselves with a discussion of the modules. As in \cite[Theorem 2.15]{Rieffel1} we can show that the functors for different $J$ satisfy ${\mathcal D}_K\circ{\mathcal D}_J = {\mathcal D}_{J+K}$.


 \section{Parametrised strict deformation quantization of
 Hilbert $C^*$-modules over $C^*$-bundles}

As an application of the general procedure 
outlined in \S\ref{sec:category}, we extend the parametrised strict deformation quantization of
$C^*$-bundles $A(X)$ to  Hilbert $C^*$-modules over $A(X)$, but somewhat more directly.

 Recall that Hilbert $C^*$-modules generalise the notion of a Hilbert space, in that they endow a linear space with an inner product which takes values in a $C^*$-algebra. They were developed by Marc Rieffel in \cite{Rieffel5},
 which used Hilbert $C^*$-modules to construct a theory of induced representations of $C^*$-algebras. In \cite{Rieffel3},
 Rieffel used Hilbert $C^*$-modules to extend the notion of Morita equivalence to C*-algebras.
 In \cite{Kas80},  Kasparov used Hilbert $C^*$-modules in his formulation of bivariant K-theory.  Hermitian vector bundles are examples of Hilbert $C^*$-modules over
 commutative $C^*$-algebras.

Let $A(X)$ be a $C^*$-algebra bundle over $X$, and $E(X)$ a Hilbert $C^*$-module over $A(X)$
respecting the fibre structure.
That is, $E(X)$ has an $A(X)$-valued inner product,
$$
\langle.,.\rangle : E(X) \times E(X) \longrightarrow A(X),
$$
and the fibre $E(X)_x: E(X)/I_x$ is a (left) Hilbert $C^*$-module over $A(X)_x$ for all $x\in X$.

Suppose now that $A(X)$ has a fibrewise action $\alpha$
of a torus $T$ and let $E(X)$ have a compatible
 fibrewise action of  $T$.

We have the direct sum decomposition,
$$
\begin{array}{lcl}
E(X)  &\cong & \widehat\bigoplus_{\chi \in \hat T} E(X)_\chi
 \\
\psi(x) &=& \sum_{\chi \in \hat T} \psi_\chi(x)
\end{array}
$$
for $x\in X$, where for $\chi\in\widehat{T}$,
$$
E(X)_\chi:= \left\{\psi\in E(X) \mid
  \alpha_t(\psi)=\chi(t)\cdot \psi\,\,
 \,\,\,   \forall\, t\in T \right\}.
$$
Since $T$ acts by $\star$-automorphisms, we have
\begin{equation}
  \label{eq:Fell_bundle_algebra2}
  A(X)_\chi\cdot E(X)_\eta \subseteq E(X)_{\chi\eta}
  \quad {\textrm and} \quad E(X)_\chi^*=E(X)_{\chi^{-1}}
  \qquad  \forall\, \chi,\eta\in\widehat{T}.
\end{equation}
Therefore the spaces \(E(X)_\chi\) for \(\chi\in\widehat{T}\) form a
{Fell bundle}~\(E(X)\) over~\(\widehat{T}\)
(see~\cite{fell_doran}); there is no continuity condition
because~\(\widehat{T}\) is discrete.
The completion of the direct sum is explained as before.
The representation theory of $T$ shows that
\(\bigoplus_{\chi\in\widehat{T}} E(X)_\chi\) is a
$T$-equivariant dense subspace
of~\(E(X)\), where $T$ acts on $ E(X)_\chi$ as follows:
$\hat \alpha_t(\phi_\chi(x)) = \chi(t) \phi_\chi(x)$ for all $t\in T, \, x\in X$.
Then $\widehat\bigoplus_{\chi \in \hat T} E(X)_\chi$
is the completion in the Hilbert $C^*$-norm of $E(X)$.

The action of $A(X)$ on $E(X)$ then also decomposes as,
$$
(\phi \psi)_\chi(x)= \sum_{\chi_1\chi_2=\chi} \phi_{\chi_1}(x)\psi_{\chi_2}(x)
$$
for $\chi_1,\chi_2, \chi \in \widehat T$, $\phi \in A(X)$ and $\psi \in E(X)$.

   Let  $\sigma\in C_b(X, Z^2(\ft{T},\torus))$ be a deformation parameter.
Then as in section 2, the parametrised strict deformation quantization of $A(X)$ can be
defined, and is denoted by $A(X)_\sigma$.

The action of $A(X)_\sigma$ on $E(X)$ can be defined by deforming the action  of $A(X)$ on $E(X)$,
$$
(\phi\star_\sigma \psi)_\chi(x) = \sum_{\chi_1\chi_2=\chi} \phi_{\chi_1}(x)\psi_{\chi_2}(x)
\sigma(x; \chi_1, \chi_2)
$$
for $\chi_1,\chi_2, \chi \in \widehat T$, $\phi \in A(X)_\sigma$ and $\psi \in E(X)$.
This is an action because $\sigma(x:, \cdot, \cdot)$ is a 2 cocycle for all $x\in X$.

This is the parametrised strict deformation quantization of $E(X)$,
denoted by $E(X)_\sigma$, which is a Hilbert $C^*$-module over $A(X)_\sigma$.

\begin{example}
Let $E\to X$ be a complex vector bundle over $X$. Consider
a smooth fiber bundle of smooth manifolds,
\begin{equation}
\xymatrix{ 
Z\ar[r]&Y\ar[d]^\pi\\
& X.
}
\end{equation}
Suppose there is a fibrewise action of a torus $T$ on $Y$. That is, assume that there is
an action of $T$ on $Y$ satisfying,
$$
\pi(t.y) = \pi(y), \qquad \forall\, t\in T, \, y\in Y.
$$
Let  $\sigma\in C_b(X, Z^2(\ft{T},\torus))$ be a deformation parameter.
$C_0(Y)$ is a $C^*$-bundle over $X$, and as in section 2, form the
parametrised strict deformation quantization $C_0(Y)_\sigma$.

Then sections $C_0(Y, \pi^*(E))$ is a Hilbert $C^*$-module over
the $C^*$-bundle $C_0(Y)$ over $X$, with a fibrewise $T$-action
compatible with the  fibrewise action of $T$ on $C_0(Y)$. Therefore
as above, we can construct a parametrised strict deformation quantization of  $C_0(Y, \pi^*(E))$,
denoted by $C_0(Y, \pi^*(E))_\sigma$, which is a Hilbert $C^*$-module over $C_0(Y)_\sigma$.

\end{example}

\bigskip 

\noindent{\bf Acknowledgements}. V.M. thanks the Australian Research Council for support, and
also thanks the hospitality at Oxford University where part of this research was completed.

\end{document}